\documentclass[11pt,reqno]{amsart}

\usepackage{amsmath}
\usepackage{amsthm}
\usepackage{amsfonts}
\usepackage{amssymb}
\usepackage{color}
\usepackage{graphicx}
\usepackage{hyperref}
\usepackage[margin=1.0in]{geometry}
\usepackage{caption}
 \usepackage{cancel}

\newtheorem{lemma}{Lemma}

\newtheorem{theorem}{Theorem}

\newtheorem{rem}{Remark}

\newcommand{\p}{\partial}

\begin{document}

\title{Gaussian curvature and gyroscopes}
\author[G. Cox]{Graham Cox$^\ast$}\email{gcox@mun.ca}
\address{Department of Mathematics and Statistics, Memorial University of Newfoundland, St. John's, NL A1C 5S7, Canada}
\author[M. Levi]{Mark Levi$^\ast$}\email{levi@math.psu.edu}
\address{Penn State University Mathematics Dept., University Park, State College, PA 16802}
\thanks{$^\ast$Both authors gratefully acknowledge support from the NSF grant DMS--0605878}
\maketitle

\begin{abstract}
We relate Gaussian curvature to the gyroscopic force, thus giving a mechanical interpretation of the former and a geometrical interpretation of the latter. We do so by considering the motion of a spinning disk constrained to be tangent to a curved surface. It is shown that the spin gives rise to a force on the disk which is equal to the magnetic force on a point charge moving in a magnetic field normal to the surface, {\it  of magnitude equal to the Gaussian curvature, and of charge equal to the disk's axial spin}. In a special case, this demonstrates that the precession of Lagrange's top is due to the curvature of a sphere determined by the parameters of the top.

\end{abstract}

 
\section{Introduction}

Lagrange's top is an axisymmetric rigid body with one point on the axis fixed, subject to a gravitational force \cite{KS08}. It is equivalent to a disk mounted on an axle attached to a ball joint, as in Figure~\ref{fig:disk}(a), or a particle---namely, the disk's center---moving on a sphere under the influence of gravitational and magnetic forces \cite{CB15}. It is the ``magnetic force," i.e. the gyroscopic effect, that causes the top to precess.

\begin{figure}[thb]
 	\captionsetup{format=hang}
	\center{ \includegraphics{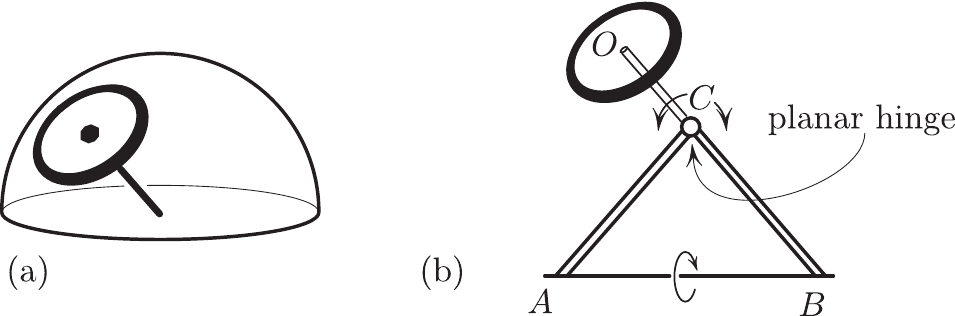} }
	\caption{(a) A spinning top viewed as a disk tangent to a sphere. (b) A disk tangent to a torus. The axle $CO$, constrained to the plane $ABC$, sweeps out a circle in that plane; the torus is generated by revolving this circle about the axis $ AB$.}
	\label{fig:disk}
\end{figure}

In this paper we point out a connection with Gaussian curvature, by showing that the motion of the Lagrange top is the same as the motion of a charged particle in a magnetic field normal to the sphere and equal in magnitude to the Gaussian curvature, with the particle's charge given by the angular momentum of the top around its symmetry axis. 

One might argue that this is just a coincidence, since the geometry of the sphere is so particular. We show that this is no coincidence, and is in fact a reflection of a deeper relationship between curvature and dynamics. To see this, we consider the generalization of the top, namely a spinning massive disk constrained to be tangent to a smooth surface, with the contact point free to slide along the surface, as illustrated in
 Figures~\ref{fig:disk}(b) and \ref{fig:heuristic.general}.
 

\begin{figure}[thb]
 \captionsetup{format=hang}
	\center{  \includegraphics{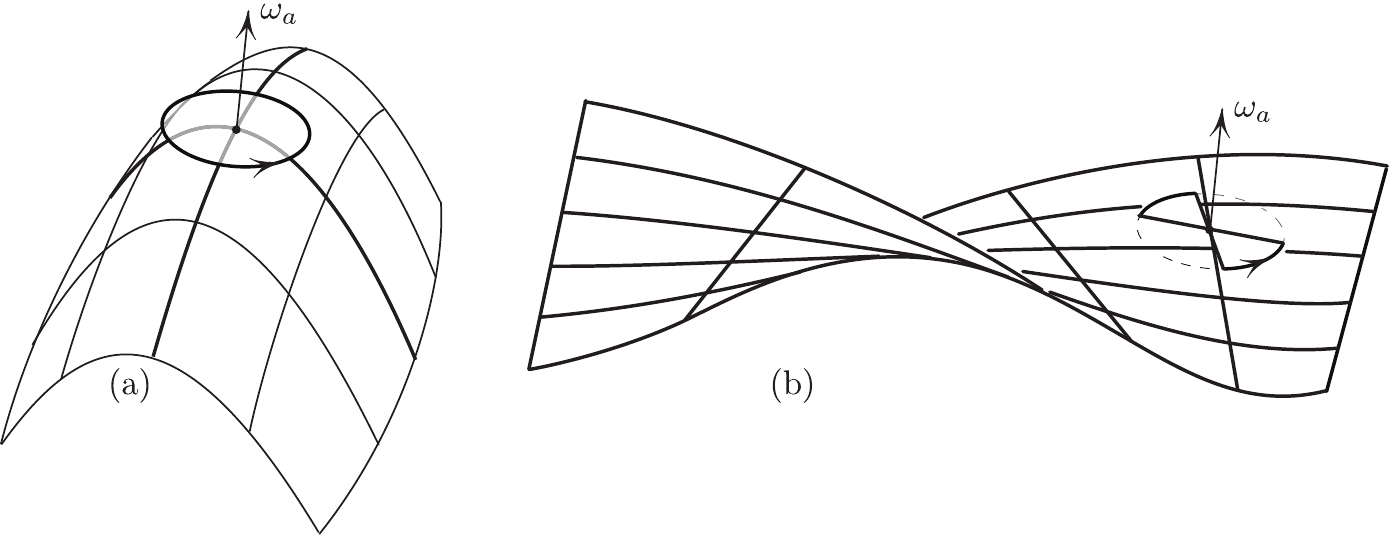}}
	\caption{The rotating tangent disk. The disk intersects the surface in the negative curvature case (b), but this situation is physically realizable, e.g. by the construction of Figure~\ref{fig:disk}(b).}
	\label{fig:heuristic.general}
\end{figure} 
\vskip 0.1 in 
\noindent {\bf  Our main result,} illustrated in Figure~\ref{fig:magforce}, is that the rotation of the disk around its axle produces an additional force acting on the center of the disk, perpendicular to the velocity, of magnitude
\begin{equation} 
	F = LvK
	\label{eq:mainresult}
\end{equation} 
where $ L $ is the axial angular momentum of the disk, $K$ is the Gaussian curvature of the surface, and $v$ is the speed of the disk's center. In other words, the trajectory of the spinning disk's center is the same as that of a non-spinning disk subject to this additional force. This is exactly the same as the Lorentz force acting on a charged particle in  a magnetic field, with $L$ being the charge and $K$ being the strength of the magnetic field. This results in a physical interpretation of Gaussian curvature as a ``magnetic field" and angular momentum as ``charge."

It should be noted that a non-spinning disk is different from a point mass, since the former has some kinetic energy due to the rotation of the tangent plane. However, if the disk has a small radius this portion of energy is small relative to the translational energy, and if it is neglected then the two systems are equivalent. By contrast, if the disk spins rapidly the kinetic energy due to the axial spin cannot be ignored, even when the radius is small. This is discussed in more detail in Section \ref{sec:smalldisk}.

We also point out that the product $ LvK $  is the only homogeneous expression in $L$, $v$ and $K$ that has units of force. 
 \vskip 0.1 in

In the approximation mentioned in the preceding paragraph, that is, for a small, rapidly-spinning disk, there exist coordinates $(x_1,x_2)$ at any given point in which the equations of motion take the form
\[
\begin{split} 
	m \ddot x_1 =- L \dot x_2 K   \\
	m \ddot x_2 = \    L \dot x_1 K. 
\end{split}
\]
 
The result is precisely formulated in Section \ref{sec:results} and proved in Section \ref{sec:classical}, with a heuristic explanation given in Section \ref{sec:heuristic}. We conclude in Section \ref{sec:coordinate.free} by giving a coordinate-free formulation and proof of the main result using the language of modern differential geometry. This provides as a special case an (apparently new) variational derivation of the equations of motion of Lagrange's top.

 \begin{figure}[thb]
 	\captionsetup{format=hang}
	\center{  \includegraphics{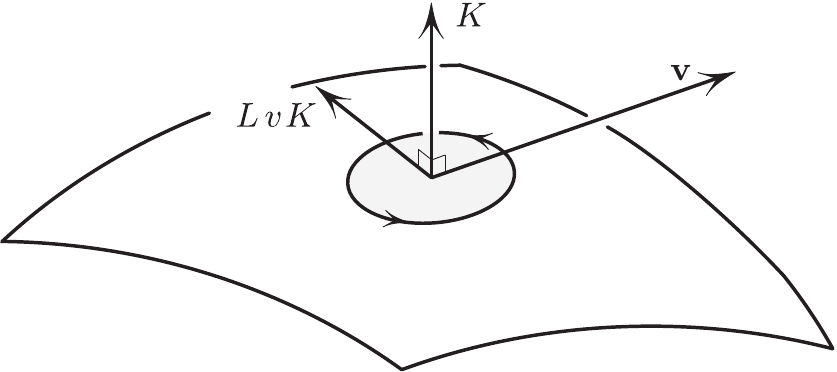}}
	\caption{The main result:   the spinning disk acts as if subject to the ``Lorentz force" $ ${\bf F}$ =  L {\bf v} \times (- {\bf K})$, with ${\bf K}= K {\bf n} $, where ${\bf n}$ is the unit normal vector in the positive direction relative to a chosen orientation of the surface and $K$ is the Gaussian curvature. Thus ${\bf B} =  - {\bf K} $ can be interpreted as the magnetic field, and $q=L$ as the charge. Choosing the opposite orientation will not affect  ${\bf F}$ as both ${\bf n}$ and $L$ would change sign. 
	}
	\label{fig:magforce}
\end{figure} 
\vskip 0.1 in 
\begin{rem}  
The result gives further motivation to the problem of studying curves for which the geodesic curvature is a constant multiple of the Gaussian curvature: $ k = cK $. Indeed, the geodesic curvature is $ k = a/v ^2 $, where $a$ is the acceleration due to the magnetic force. Substituting $ a = m ^{-1} LvK$ from \eqref{eq:mainresult}, we get 
\[
	k = \frac{L}{mv} K.
\] 
Note that $ v$ is constant since the force is perpendicular to the velocity, and the constant $c$ is precisely the ratio of angular momentum to linear momentum. 
This is a subclass of a wider class of problems on ``magnetic geodesics" considered by Arnold \cite{A88} and many others; see \cite{GG09} and references therein.
\end{rem}

\vskip 0.1 in 
\begin{rem} 
We arrived at the spinning disk problem through the following series of associations. In analyzing the dynamics of a tight binary orbiting a larger star we were led naturally to a simplified problem: a dumbbell, such as a pair of tethered satellites, spinning and orbiting a gravitational center (this work is to appear elsewhere). This in turn led us to consider a ``geodesic dumbbell":  two point masses constrained to a surface and a fixed small geodesic distance apart. An interesting effect in all of these problems is the appearance of a ``magnetic" force due to the spin. (Incidentally, such a ``magnetic" force appears also for a single point mass if the force field itself is spinning, as pointed out in \cite{KL2015}). The spinning disk considered here is a simplification of the geodesic dumbbell, chosen to demonstrate the magnetic effect with a minimum of technicalities.
\end{rem} 
\section{Results}
\label{sec:results}

We now formulate the main result of the paper---the equations of motion of the spinning disk on a curved surface. In order to state the result we first write the kinetic energy of the disk in a special coordinate system, then interpret geometrically the resulting Euler--Lagrange equations.

\subsection{Classical formulation}
Let $ (x_1, x_2) $ be a local rectangular coordinate system on the surface,  so that the lines 
$\{x_1 = {\rm const.}\}$ are orthogonal to the lines $\{x_2 = {\rm const.}\}$. 
We take these coordinates to be right-handed with respect to a chosen orientation of the surface. 
The metric $ds^2$ on the surface is then given by 
\begin{equation} 
	ds^2= a_{11} dx_1^2 +a_{22} dx_2^2.
	\label{eq:metric}
\end{equation} 
Let $x=(x_1,x_2)$ denote the coordinates of the disk's center.

The rotational part of the kinetic energy of the disk is $\frac{1}{2} \left< I \boldsymbol{\omega}, \boldsymbol{\omega} \right>$, where $I$ is the tensor of inertia of the disk around its center of mass and $\boldsymbol{\omega}$ is  the angular velocity. Decomposing the latter along the axial direction and the rest, we obtain the total kinetic energy (translational plus rotational) as 
\begin{equation} 
	E = \frac{1}{2} m v ^2 + \frac{1}{2} I_a \omega _a ^2 +\frac{1}{2} I_d \omega _d ^2,
	\label{eq:Edisk}
\end{equation} 
where $ \omega _a $ and $ \omega _d$ are the scalar values of the two projections of $\boldsymbol{\omega}$ (axial and along a diameter), $m$ is the mass of the disk, and $v  $ is the speed of the center of the disk: $ v ^2 = \langle A \dot x, \dot x\rangle $, where 
$ A = {\rm diag}(a_{11}, a_{22}) $ and where $ \langle\cdot\,, \cdot \rangle $ denotes the Euclidean inner product. It follows that a non-spinning\footnote{I.e. $\omega_a=0$. Geometrically this means any material radius-vector of the disk   undergoes parallel transport along the curve $x(t)$.} disk has kinetic energy
\begin{equation} 
	T = \frac{1}{2} m v ^2 + \frac{1}{2} I_d \omega _d ^2, 
	\label{eq:Tdisk}
\end{equation} 
and its center moves according to
\[
	\frac{d}{dt} T_{\dot x} - T_x =0.
\]

We now describe what happens when the disk has non-zero angular momentum around its axle. 

\begin{theorem}\label{thm:coordinate}
Define $T = T(x,\dot x)$ as in \eqref{eq:Tdisk} and let $K = K(x)$ denote the Gaussian curvature of the surface. Then the angular momentum $L = I_a \omega_a$ about the disk's axle is a conserved quantity and the motion of the disk's center is governed by 
\begin{equation} 
	\frac{d}{dt} T_{ \dot x }-T_x=\sqrt{ a_{11}a_{22} } L (J \dot x)K, \  \  J = \begin{pmatrix} 0 & -1 \\ 1 & 0 \end{pmatrix}.
	\label{eq:eleq}
\end{equation}    
\end{theorem}

To uncover the intrinsic meaning of the equations of motion (\ref{eq:eleq}), fix a point $P$ on the trajectory and let   $(x_1,x_2)$ be the Cartesian coordinates in the plane tangent to the surface at $P$; the nearby positions on the trajectory are then given by the coordinates $(x_1,x_2)$ of the orthogonal projection onto this tangent plane.  
  In these coordinates,   $ a_{11}=a_{22}=1$ at $P$, and \eqref{eq:eleq} simplifies to
  \begin{equation} 
	\frac{d}{dt} T_{ \dot x }-T_x=L(J \dot x)K  \  \  \  \  \hbox{at}  \  \  P.
	\label{eq:eleq1}
\end{equation}   
This confirms our earlier claim that the axial spin amounts precisely to the addition of a ``magnetic field" orthogonal to the surface and of magnitude $K$.

\begin{rem}
It is possible to include potential energy terms in $T$, allowing for external forces such as gravity, but these have no bearing on our result as they do not affect the ``magnetic" term. 
\end{rem}

\subsection{The small, rapidly-spinning disk}\label{sec:smalldisk}
Further simplification occurs if we ignore the last term in \eqref{eq:Edisk}. This approximation is justified if the radius of the disk is small in the following sense. Since $|\omega_d| \leq \|h\| v$, where $h$ denotes the second fundamental form of the surface, we obtain
\[
	\frac{I_d \omega_d}{mv^2} \leq \frac{R^2 \|h\|^2}{4}
\]
where $R$ is the radius of the disk. Thus the translational kinetic energy dominates the energy of rotation around the diameter when $R\|h\| \ll 1$. Moreover, since $I_a = 2I_d$, the second term in \eqref{eq:Edisk} will be much larger than the third if $\omega_a$, the axial spin, is large relative to $\omega_d$, i.e. $\omega_d/\omega_a \ll 1$.

We thus define the \textit{small, rapidly-spinning disk} to be the system with kinetic energy
\begin{equation}
	E = \frac{1}{2} m v ^2 + \frac{1}{2} I_a \omega _a ^2.
	\label{eq:E1disk}
\end{equation}
In this case the center moves {\it precisely} like a particle of charge $ L$ in a magnetic field of strength $ K $. In particular, \eqref{eq:eleq1} becomes
\begin{align}
	m \ddot x = L (J \dot x) K  \  \  \  \  \hbox{at}  \  \  P,  
\end{align}
or equivalently
\begin{align}
\begin{split}\label{eq:P}
	m \ddot x_1 =- L\dot x_2 K     \\
	m \ddot x_2 = \  L\dot x_1 K  . 
\end{split}
\end{align}



\subsection{An application to the Lagrange top} We now describe more explicitly the role of Gaussian curvature in the dynamics of the Lagrange top.  

The configuration space $SO(3)$ can be parametrized by Euler angles, which we denote by $x_1$, $x_2$ and $\theta$, as shown in Figure~\ref{fig:eulerangles} (thus breaking with the tradition of using Greek letters exclusively). Here $(x_1, x_2)$ are the coordinates of the point of intersection of the axle with a sphere,
and $\theta $ is the angle between the parallel on the sphere through $P$ and a vector ${\bf u}$ rigidly attached to the top and normal to the top's  axle. The kinetic energy is
\begin{equation} 
	E_{\rm top} = \frac{1}{2} I_1 (\dot x _1 ^2 + \dot x _2^2 \sin^2 x_1) + \frac{1}{2} I_3(\dot \theta    + \dot  x_2   \cos x_1  )^2,
	\label{eq:etop}
\end{equation}   
where $I_3$ is the moment of inertia about the top's axle, and $I_1$ is the moment of inertia about any perpendicular axis. The first term in \eqref{eq:etop} coincides with  the expression for the kinetic energy  of a point mass $m$ on the sphere of radius $R$, namely 
$ \frac{1}{2} m R^2 (\dot x _1 ^2 + \dot x _2^2 \sin^2 x _1)$, provided  
\begin{equation} 
	m R ^2 = I_1 . 
	\label{eq:mr2}
\end{equation}
This still leaves some freedom in choosing $m$ and $R$. We use this freedom by insisting that the potential energy (which we excluded so far) of our point mass equal the potential energy of the top: $mg R \cos x_1=Mg \ell \cos x_1$, where $M$ denotes the top's mass, and $\ell$ the distance from the origin to the top's center of mass. This and   (\ref{eq:mr2})  fixes $R$ and $m$: 
\begin{equation}\label{Rdef}
	R = \frac{I_1}{M \ell}, \quad m = \frac{M^2 \ell^2}{I_1}.
\end{equation}
To summarize, to any Lagrange top we assign a radius $R$ and a point mass $m$, so the kinetic energy of the top becomes  
 \begin{equation} 
	E_{\rm top} = 
	\frac{1}{2}m v^2 + \frac{1}{2} I_3 \omega_a ^2 . 
	\label{eq:etop1}
\end{equation}   
This  is precisely the form of \eqref{eq:E1disk} with $I_a = I_3$. That is, the Lagrange top turns out to be an exact incarnation of the ``small, rapidly spinning disk" moving on a sphere of radius $R$, as given by \eqref{Rdef}.
According to Theorem \ref{thm:coordinate}, the  point at which the axle punctures the sphere behaves as if it were subject to a Lorentz force of magnitude $L v  K$, where $L$ is the axial angular momentum and where $ K = 1/R^2$ is the Gaussian curvature of the sphere.

 \begin{figure}[thb]
 	\captionsetup{format=hang}
	\center{  \includegraphics{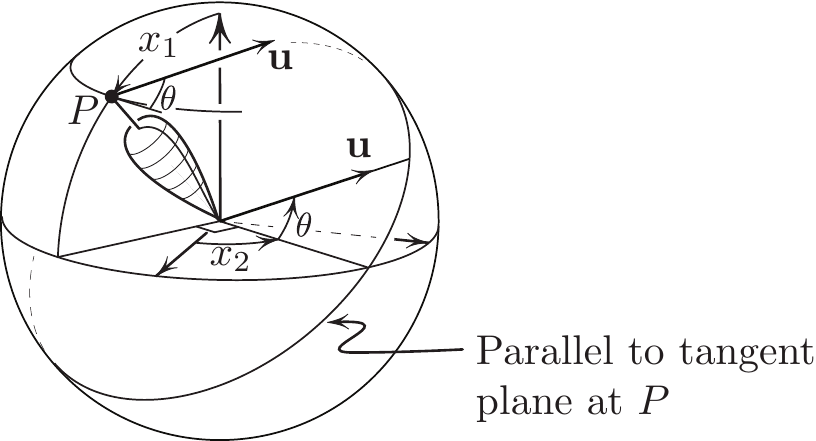}}
	\caption{Euler's angles: $ x_1 = \pi/2\,-\, $latitude; $ x_2=\pi/2\,-\,$longitude; and $ \theta = $ the angle of a chosen vector ${\bf u}$   with the parallel $\{x_1 = {\rm const.}\}$. }
	\label{fig:eulerangles}
\end{figure}

\subsection{Intrinsic formulation}\label{sec:intrinsic}
%
%
%


The equations of motion \eqref{eq:eleq} depend on local coordinates, which tends to obscure their geometric nature. On the other hand, the simplified forms \eqref{eq:eleq1} and \eqref{eq:P}, while geometrically transparent, are only valid in a particular coordinate system at a single point. We address this concern by writing the equations of motion in an invariant manner, retaining the special assumption of a small, rapidly-spinning disk.

The center of the disk moves along a parametrized curve $\gamma$ on the surface, with velocity $\dot\gamma$. Letting $D_t$ denote the covariant derivative operator along $\gamma$, the acceleration of the center is $D_t \dot\gamma$, and so the equation of a geodesic is $D_t \dot\gamma = 0$. In these notations Theorem \ref{thm:coordinate} can be reformulated in the following coordinate--free way.

\begin{theorem}\label{thm:coordinate.free}
The center of the small, rapidly-spinning disk moves according to
\begin{align}\label{eqn:intrinsic}
	m D_t \dot\gamma &= L (J \dot\gamma)K,
\end{align}
where $L$ is constant and $J$ denotes counterclockwise rotation by $ \pi/2 $ in the tangent plane.
\end{theorem}

This should be read as a coordinate-free version of \eqref{eq:P}: the left-hand side is mass times acceleration and the right-hand side is the magnetic force acting on the disk's center. Further details and the proof are given in Section \ref{sec:coordinate.free}.

\section{The equations of motion: coordinate version}
\label{sec:classical}
In this section we   show how the Gaussian curvature arises in the Euler--Lagrange equations, thus proving Theorem 1. 

In order to write the kinetic energy $E$ in coordinates, we choose a local rectangular coordinate system $ (x_1, x_2) $ on the surface. Marking a particular radius on the disk, denote by $\theta$ the angle this radius makes with the positive direction of the lines $ \{x_2 = {\rm const.} \} $. Thus the triple $ (x_1, x_2, \theta)$ parametrizes (locally) the configuration space of the disk. Having thus coordinatized the configuration space, we need to find an expression for  $\omega _a$.  To that end we need an expression of the parallel transport (which corresponds to $ \omega_a=0 $); this is provided by the following lemma.

\begin{lemma}\label{lem:pt}
The parallel transport of a tangent vector along a curve $x(t)=(x_1 (t),x_2(t))$ on the surface is given by 
\begin{equation} 
	  \dot    \theta   = - \sqrt{ a_{11} } k_ 1 \dot x_1(t)  -  \sqrt{ a_{22}}k_2  \dot x_2 (t) , 
	\label{eq:paralleltransport}
\end{equation} 
where $k_ 1$ and $k_2$ are the geodesic curvatures of the coordinate lines $\{x_2= {\rm const.}\}$ and $\{x_1= {\rm const.}\}$, respectively, and $ a_{ii} $ are the coefficients of the metric tensor   (\ref{eq:metric}). 
\end{lemma}

\begin{rem}
A more geometric form of   (\ref{eq:paralleltransport}) (which we do not use) is 
\begin{equation} 
	d\theta = -k_1\,ds_1-k_2\,ds_2, 
	\label{eq:paralleltransport1}
\end{equation}   
where $ ds_i $ are the length elements along the coordinate curves. Equivalently, 
\begin{equation} 
	 \theta ^\prime  =-k_1\cos \varphi -k_2\sin \varphi, 
	\label{eq:paralleltransport2}
\end{equation}  
where $ \varphi $ is the angle between the curve $x(t)$ and the coordinate lines $\{x_2 = {\rm const.}\}$ 

\end{rem}
 \begin{figure}[thb]
 	\captionsetup{format=hang}
	\center{  \includegraphics{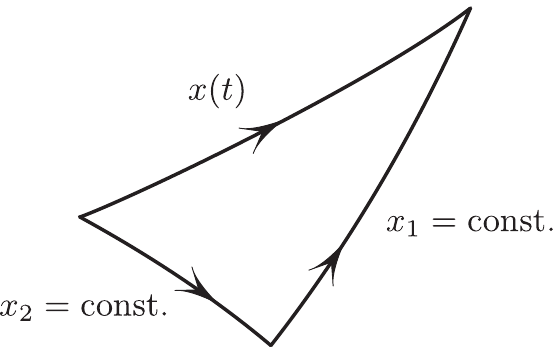}}
	\caption{Towards the proof of   (\ref{eq:paralleltransport}).}
	\label{fig:paralleltransport}
\end{figure}

%

\begin{proof}   Figure \ref{fig:paralleltransport} shows an infinitesimal triangle whose two sides are lines $ x_1 = {\rm const.} $, $ x_2 = {\rm const.} $,  
and whose ``hypotenuse" is a segment of $ x(\cdot) $  along which we want to transport. Transporting a tangent vector along the curve $ x(\cdot) $ 
we get $ \Delta _{hyp}\theta $ (the change of the angle it forms with the lines $ x_2 = {\rm const.} $,  the quantity we are interested in). But by the Gauss--Bonnet theorem  
\[
	\Delta _{hyp}\theta = \Delta_1\theta + \Delta _ 2 \theta + O(dx_1 dx_2),   
\] 
where $ \Delta_i \theta $ denote  the change of $\theta$ under parallel transport along the ``legs" of the triangle in 
Figure~\ref{fig:paralleltransport} in the direction of the arrows.   Using 
\[
	\Delta_ 1 \theta = -k_1\,ds_1, \  \  \hbox{where}  \  \    ds_1=  \sqrt{ a_{11}}\,dx_1, 
\] 
and similarly for the other curve leads to   (\ref{eq:paralleltransport1}) and thus to (\ref{eq:paralleltransport}).
\end{proof}

 According to  (\ref{eq:paralleltransport}), the {\it  axial} component of the  angular  velocity of the disk is 
\begin{equation} 
	\omega_a=   \dot \theta + k_ 1  \sqrt{ a_{11}}  \dot x_1  + k_ 2   \sqrt{ a_{22}}  \dot x_2     \buildrel{def}\over{=}  \dot\theta + f(x) \dot x, 
	\label{eq:omega.r}
\end{equation}   
where
\[
	x= (x_1, x_2)  \  \  \hbox{and}  \  \ f(x) = (k_1 \sqrt{ a_{11}} ,  k_2 \sqrt{ a_{22}}).
\] 
Indeed, according to the lemma, the right--hand side of   (\ref{eq:omega.r}) measures the mismatch with parallel transport.  
 The kinetic energy of the disk in these coordinates is   
 \begin{equation} 
	 E=  \frac{1}{2}I_a( \dot\theta + f(x) \dot x)^2 + 
	 \underbrace{\frac{1}{2} \big( m\langle A\dot x, \dot x \rangle + I_d h(\dot x, \dot x) \big)}_{T=T(x, \dot x )},
	\label{eq:Edisk1}
\end{equation}   
where $h$ is the second fundamental form of the surface.  The term $T$ is the kinetic energy of the non-spinning disk; if $I_d h(\dot x, \dot x)$ can be neglected then $T$ is simply the kinetic energy of the point mass $m$. 

The Euler--Lagrange equations for the Lagrangian $E$  are
\begin{equation} 
	 \frac{d}{dt} ( \dot\theta + f(x) \dot x)=0
	\label{eq:el.theta}
\end{equation}   
 and
 \begin{equation} 
	 \frac{d}{dt}  T_{ \dot x }-T_x + I_a \biggl( ( \dot\theta + f(x) \dot x)(f_x-f_x^T) \dot x + \frac{d}{dt} ( \dot\theta + f(x) \dot x)\biggl) =0. 
	\label{eq:el.x}
\end{equation}   
According to \eqref{eq:el.theta}, $  \dot\theta + f(x) \dot x$ is constant along any solution---this is simply the angular velocity of the disk around its axle. Fixing an arbitrary value $ \omega _a$ of this constant, we limit our attention to the invariant subspace of the phase space satisfying 
\[
	\dot\theta + f(x) \dot x = \omega_a.
\] 
 For such solutions \eqref{eq:el.x} reduces to 
 \begin{equation} 
	\frac{d}{dt}  T_{ \dot x }-T_x + L  (f_x-f_x^T) \dot x =0, 
	\label{eq:el.x1}
\end{equation}   
where $  L=I_a \omega_a $, with the skew-symmetric matrix
 \begin{equation} 
	 f_x-f_x^T =- \underbrace{\biggl(   (k_1  \sqrt{ a_{11}})_{x_2} -(k_2  \sqrt{ a_{22}})_{x_1} 
	  \biggl)}_{\widehat{K(x)}}  \underbrace{\begin{pmatrix} 0 & -1 \\ 1 & 0 \end{pmatrix}}_{J} .
	\label{eq:B}
\end{equation}   

\begin{lemma}\label{lemma.K}
With the above assumptions and notations, 
\begin{equation} 
	\widehat{K(x)} = \sqrt{a_{11}a_{22}}\,K(x), 
	\label{eq:lemma.K}
\end{equation}   
where $K$ is the Gaussian curvature of the surface. 
\end{lemma}

\begin{proof}
Assume $ x=(0,0)$ without loss of generality. Consider a patch on the surface corresponding to 
 $ x_1\in [0,\varepsilon ] $ and $ x_2\in [0,\delta ] $ for small $\varepsilon$, $\delta$. The Gauss--Bonnet formula applied to this patch gives 
\begin{equation} 
	\int_{0}^{\varepsilon } \int_{0}^{\delta } K(x_1,x_2)  \sqrt{a_{11}a_{22}} \,dx_1\,dx_2+ \cancel{4 \frac{\pi }{2}}+ \int k\,ds =\cancel{2 \pi } , 
	\label{eq:gb}
\end{equation} 
where the term $ 4 \frac{\pi }{2} =2 \pi $ comes from the four right angles of the rectangle---here the assumption of the orthogonality of the coordinate system is used---and $k$ is the geodesic curvature. Now $ \int k\,ds $ is the sum of four terms, which we group into pairs corresponding to opposite sides:
\[
	\int k\,ds =  \int_{0}^{\varepsilon } ((a_{11}k_1)(x_1,0)- (a_{11}k_1)(x_1,\delta))  \,dx_1\ +
	\int_{0}^{\delta }((a_{22}k_2)(\varepsilon,x_2) -
	 (a_{22}k_2)(0,x_2))\,dx_2,
\] 
where $(a_{11}k_1)(x,y)$ denotes the product evaluated at $(x,y)$. Substituting this into \eqref{eq:gb}, dividing by $\varepsilon\delta$ and sending $ \varepsilon, \delta\rightarrow 0 $ results in \eqref{eq:lemma.K}.
\end{proof}

Substituting \eqref{eq:lemma.K} into \eqref{eq:B} we reduce the Euler--Lagrange equation \eqref{eq:el.x1} to the form \eqref{eq:eleq}, thus completing the proof of Theorem 1. 

\begin{figure}[thb]
	\captionsetup{format=hang}
	\center{  \includegraphics{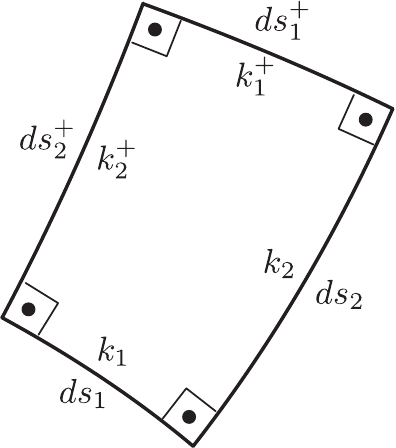}}
	\caption{Illustrating the proof of Lemma \ref{lemma.K} on Gaussian curvature.}
	\label{fig:quadrangle}
\end{figure} 
\vskip 0.1 in

\section{A heuristic explanation of the appearance of Gaussian curvature}
\label{sec:heuristic}
We now give a short heuristic explanation of our result. After that we discuss two special cases,  
stripped of all formalism, to give a physical insight into ``what is really going on." The underlying intuition in all this is the following: if we hold a spinning wheel by 
two ends of its axle and attempt to change its orientation, the axle ``resists" by pushing our hands in the direction {\it  orthogonal} to their motion. An elementary explanation of this can be found in, e.g., \cite{L14}, p. 154.     
\vskip 0.1 in 

\paragraph{\bf An explanation of the magnetic force -- general case.} 

 In the neighborhood of the tangency point with the disk the surface is given by the graph of 
\[
	z =   \frac{1}{2} \left< H{\bf x}, {\bf x} \right> + O( | {\bf x} |^3 ), \  \  H = \begin{pmatrix} a & b \\ b & c \end{pmatrix},\  \  {\bf x} = 
	\begin{pmatrix} x_1 \\ x_2 \end{pmatrix}.
\]

 A unit normal vector to this surface is $ {\bf n} = ( H {\bf x}, -1) + O( |{\bf x}|^2) \in {\mathbb R}  ^3$, 
and $\dot {\bf n}  = (H {\bf v}, 0 ) $ at the instant when the center of the disk is at $ {\bf x} = {\bf 0}$. But $I_a\omega_a {\bf n} = {\bf L} $ is the axial angular momentum,  and thus the torque upon the disk is 
\begin{equation} 
   	\mbox{\boldmath$\tau$}    =   \dot {\bf L}  =    I_a \omega_a  (H {\bf v},0)  . 
		\label{eq:tau}
\end{equation} 
This torque can be attributed to two forces, $ {\bf F}_a$ and $ {\bf F}_b$,  normal to the disk acting at two points $a$ and $b$ whose connecting segment is normal to $ H {\bf v} $, as shown in Figure~\ref{fig:general}. 
 \begin{figure}[thb]
 	\captionsetup{format=hang}
	\center{  \includegraphics{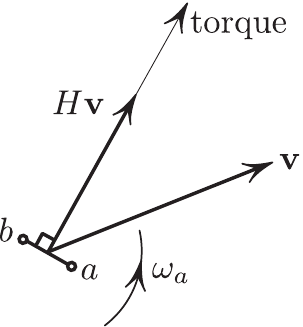}}
	\caption{The force $ {\bf F}_b $ upon the disk points out of the page; $ {\bf F}_a $ points  into the page.}
	\label{fig:general}
\end{figure} 
The resultant force is
\[
	 {\bf F}_b+ {\bf F}_a=({\bf n}_b-{\bf n} _a)F, 
\] 
where ${\bf n}_a $ and $ {\bf n}_b $ are unit normal vectors and $F= |  {\bf F}_a | = |  {\bf F}_b | $. 
The $(x_1,x_2)$-projection of the resultant is therefore 
\begin{equation} 
	({\bf n}_b-{\bf n} _a)_{x_1 x_2}F = H(b-a) F = H \frac{(H {\bf v})^\perp}{|H {\bf v} | } ds\,F = H \frac{JH {\bf v}}{|H {\bf v} | } \tau . 
		\label{eq:resultant1}
\end{equation} 
A direct computation yields $HJH = KJ$, where $K = \det H$ is the Gaussian curvature.
Substituting this and $ \tau = | \mbox{\boldmath$\tau$} | $ from \eqref{eq:tau}  into    (\ref{eq:resultant1})  we get 
 \[
	 ({\bf F}_b+ {\bf F}_a)_{x_1 x_2}=I_a \omega_a (J {\bf v})K = L (J {\bf v})K,
\] 
in agreement with \eqref{eq:mainresult} and Figure~\ref{fig:magforce}. 

The above explanation still relies on a small calculation   (\ref{eq:resultant1}); the following two special cases are treated purely geometrically, with nothing hidden by calculation. 
 
\subsection*{Special case \#1: ${\bf v}$ is tangent to a line of principal curvature, Figure \ref{fig:heuristic.ellipsoid}(a)} 
\vskip 0.1 in 

\begin{figure}[thb]
 \captionsetup{format=hang}
	\center{  \includegraphics{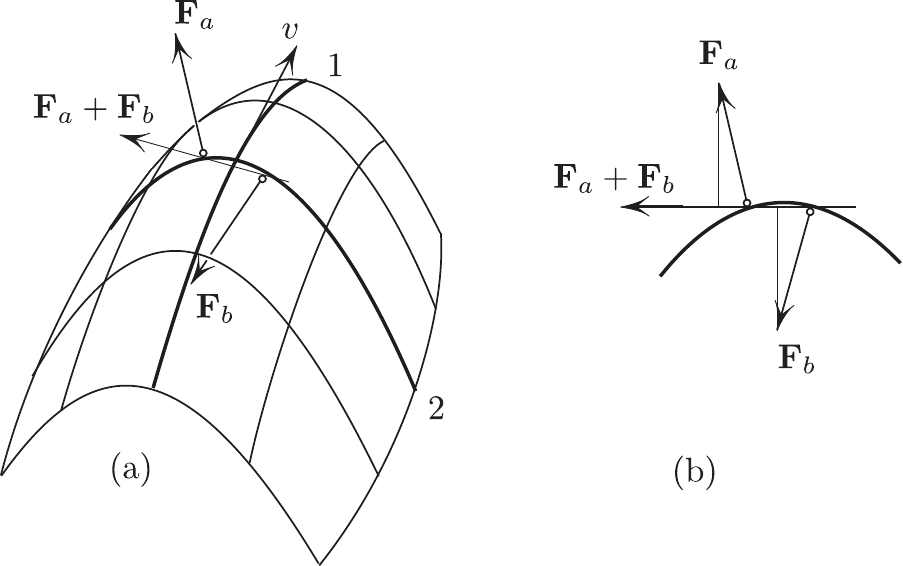}}
	\caption{The vector ${\bf v}$ is tangent to a line of principal curvature. The torque upon the disk is equivalent to the torque of a couple $ {\bf F}_a, {\bf F}_b $ exerted upon the disk at two points a small distance $ ds$ apart.}
	\label{fig:heuristic.ellipsoid}
\end{figure} 
With the velocity ${\bf v}$ pointing in  a principal direction, the disk undergoes an additional rotation around the diameter tangent to the other principal direction (line $2$ in Figure~\ref{fig:heuristic.ellipsoid}(a)) with the angular velocity $ \omega_r =v\kappa_1 $, where $ \kappa_1 $ is the principal curvature  in the direction ${\bf v}$.   As the result of this reorientation, the disk exerts the gyroscopic torque\footnote{see, e.g., \cite{L14}.}
 \begin{equation} 
	 \tau =  L v \kappa_1
	\label{eq:torque}
\end{equation} 
against the tangency constraint (this can be seen from   (\ref{eq:tau})), and the surface reacts with an equal and opposite torque, which can be considered to be due to the torque of two forces, $ {\bf F}_a $ and $ {\bf F}_b $, applied by the surface to the disk, as in Figure~\ref{fig:heuristic.ellipsoid}(a). The torque of these forces points into  the page in Figure~\ref{fig:heuristic.ellipsoid}(b). 
\begin{figure}[thb]
 	\captionsetup{format=hang} 
	\center{  \includegraphics{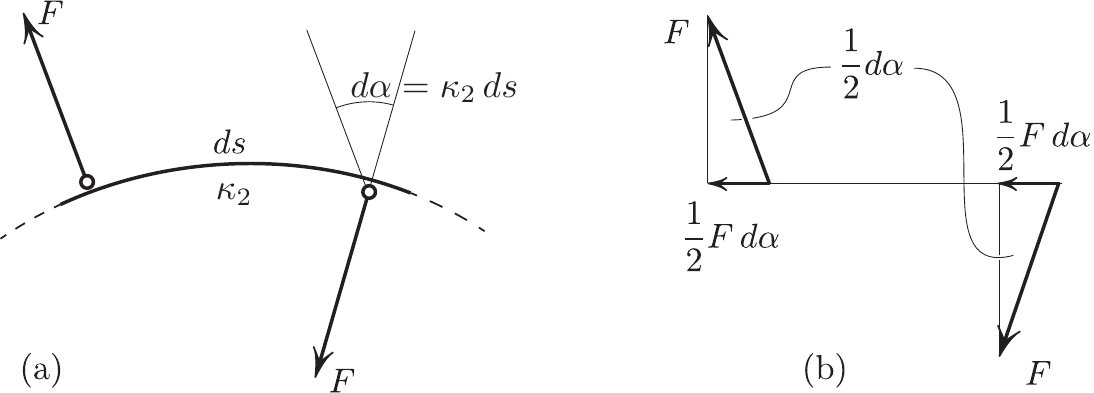}}
	\caption{The origin of the ``magnetic" force. (a) The torque $\tau$  which the surface exerts upon the disk can be equivalently represented by the torque of two forces   at an infinitesimal distance $ ds$; the result \eqref{eq:F.deflection} does not depend on $ ds$. (b) Computing $ F_{\rm deflection} $.}
	\label{fig:heuristic}
\end{figure} 
Figure~\ref{fig:heuristic.ellipsoid}(b) explains the key point: the sum of these reaction forces points to the left in the figure, i.e. is normal to ${\bf v}$. This is precisely the  ``magnetic force" mentioned earlier. 

We now compute the magnitude of this force, referring to the magnified view of Figure~\ref{fig:heuristic}(b). We have  
\begin{equation} 
	F_{\rm deflection}=F\,d \alpha = F\, \kappa_2\,ds,
	\label{eq:F}
\end{equation}  
where $\kappa_2$ is the curvature of curve $ 2 $ in Figure~\ref{fig:heuristic.ellipsoid}. But
\[
	  F\,ds = \tau \buildrel{(\ref{eq:torque})}\over{=}   L v\kappa_1.
\] 
Substituting this into \eqref{eq:F} gives 
\begin{equation} 
	F_{\rm deflection}=L v\kappa_1\kappa_2 = LvK, 
	\label{eq:F.deflection}
\end{equation} 
which agrees with Figure~\ref{fig:magforce} and coincides with the general formula derived in Sections \ref{sec:classical} and \ref{sec:coordinate.free}.
 
As an example, for the cylinder or for the cone we have $ F_{\rm deflection}=0 $. In general the sign of $K$  affects the direction (``left" or ``right" of ${\bf v}$) of the deflection force. 
\vskip 0.2 in 

\subsection*{Special case \#2:  ${\bf v}$ points along a straight line on a ruled surface of negative curvature}   

 \begin{figure}[thb]
 	\captionsetup{format=hang}
	\center{  \includegraphics{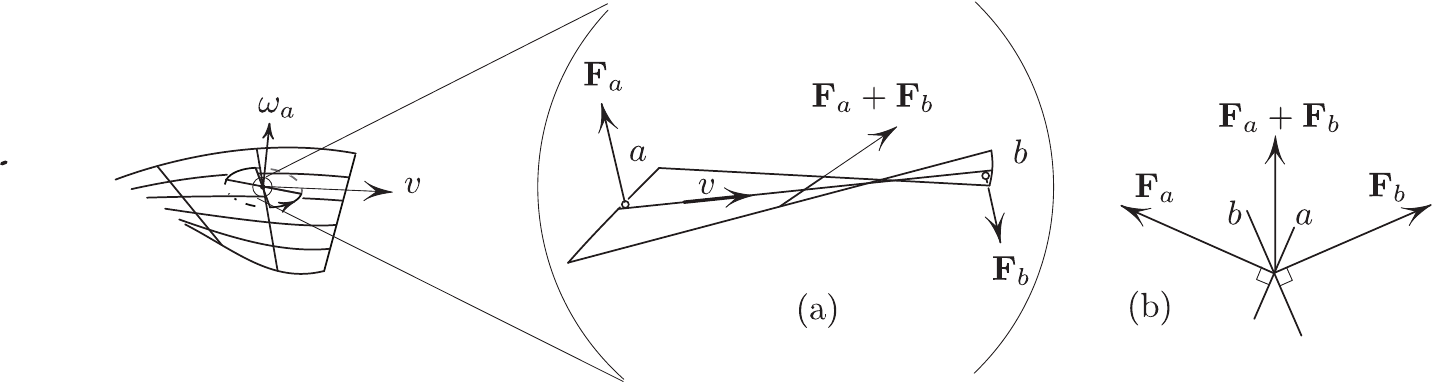}}
	\caption{The velocity ${\bf v}$ points along the straight line on a ruled surface of negative curvature. In (b), ${\bf v}$ points into the page.}
	\label{fig:heuristic.saddle}
\end{figure} 
When ${\bf v}$ points along a straight geodesic (Figure~\ref{fig:heuristic.saddle}), the disk's plane undergoes instantaneous rotation around the diameter aligned with this geodesic (in addition to the axial spin). The resulting gyroscopic effect causes the disk to apply a torque to the  surface; the surface reacts with the equal and opposite   torque, which can be thought of being due  two  point forces $ {\bf F}_a $ and ${\bf F}_b $ acting at
two points $a$ and $b$ on the geodesic (this torque tries to twist the disk around the diameter perpendicular to ${\bf v}$). Figure~\ref{fig:heuristic.saddle}(b) shows that these forces have a nonzero resultant, {\it  due to the surface's twisting}, and  that this resultant points normal to  ${\bf v}$. It remains to confirm that the magnitude of this force is given by (\ref{eq:F.deflection}).

Let us choose coordinates so that the straight geodesic is the $x$-axis, and the $z$-axis is normal to the surface at the point in question. Any ruled surface satisfies
\[
	z= \kappa xy + o(x ^2 + y ^2 )
\]    
for some constant $\kappa$. The Gaussian curvature of this surface at $(0,0,0) $  is   $ K = - \kappa ^2 $.  
We are interested in the resultant $ {\bf F}_a +{\bf F}_b $; its magnitude is  
\begin{equation} 
	F \,d\alpha = F \kappa \,dx, 
	\label{eq:resultant}
\end{equation}   
where $ F= | {\bf F}_a|=| {\bf F}_a|$, $\alpha$ is the angle between the normals at $a$ and $b$ and $x$  is the distance along the geodesic. We find $F$ from the torque condition, as we did before in   (\ref{eq:torque}): 
\[
	\tau = F\,dx=L \omega_r=L \kappa v. 
\] 
We conclude that the resultant   (\ref{eq:resultant})  is  
\[
	F \kappa \,dx = L \kappa ^2  v = - L  Kv, 
\] 
in agreement with the general case, as claimed.

\section{The equations of motion: coordinate-free version}
\label{sec:coordinate.free}
 In this section we outline an alternative   proof of Theorem \ref{thm:coordinate.free}, and thus of Theorem \ref{thm:coordinate}. As opposed to the proof given above, this one is variational and coordinate--free. 
 
 As in Section \ref{sec:intrinsic}, we suppose the center of the disk follows a curve $\gamma$ on the surface, with velocity $\dot\gamma$ and acceleration $D_t\dot\gamma$. We fix a   unit vector $u$ pointing from the center of the disk to the distinguished point on the boundary, so the configuration space of the spinning disk is the unit tangent bundle of the surface.

By the argument of Lemma \ref{lem:pt}, the axial component of the angular velocity is $\omega_a = \left<D_t u, Ju\right>$, and hence the small, rapidly-spinning disk has kinetic energy
\begin{align}
	E(\gamma,u) = \frac12 m|\dot\gamma|^2 + \frac12 I_a |D_t u|^2.
\end{align}
 
For a smooth curve $(\gamma,u)$   in the unit tangent bundle\footnote{That is, $\gamma$ is a curve on the surface and $u$ is a curve of unit tangent vectors such that $u(t)$ is in the tangent plane at $\gamma(t)$ for each time $t$.} we define the action integral
\[
	\mathcal{S}(\gamma,u) = \frac12 \int_a^b \left(m |\dot \gamma|^2 + I_a |D_t u|^2 \right) dt.
\]
The desired equations of motion are precisely the variational (Euler--Lagrange) equations for $\mathcal{S}$.

To compute the variation, let $(\gamma_s(t), u_s(t))$ denote a smooth family of curves in the unit tangent bundle, with  
$(\gamma_0(t), u_0(t))=
(\gamma(t), u(t)) $.   Define the variation fields $V = \p_s \gamma_s$ and $w = D_s u$.
Substituting $ \gamma_s, \  u_s $ into the action integral and differentiating the first term in the integrand,  we have
\begin{align}\label{CMvar}
	\frac12 \frac{d}{ds} |\dot\gamma_s|^2 = \left<D_s \dot\gamma, \dot\gamma\right>
	= \left<D_t V, \dot\gamma\right>.
\end{align}

For the second term in  we compute at $ s=0 $: 
\[
	\frac12 \frac{d}{ds} |D_t u_s|^2 = \left<D_s D_t u, D_t u\right> 
	= \left<D_t D_s u, D_t u\right> + \left<R(V,\dot\gamma)u, D_t u\right>,
\]
where $R$ is the Riemann curvature operator.  From the definition of $w$ we obtain $\left<D_t D_s u, D_t u\right> = \left<D_t w, D_t u\right>$. Next, using the formula for the Riemann curvature tensor on a surface  we compute  
\begin{align*}
	\left<R(V,\dot\gamma)u, D_t u\right> 
	&= K \left( \left<V, D_t u\right>\left<\dot\gamma, u\right> - \left<V, u\right>\left<\dot\gamma, D_t u\right> \right) \\
	&= \big<V, K \left( \left<\dot\gamma, u\right>D_t u - \left<\dot\gamma, D_t u\right>u \right)\big>.
\end{align*}
Substituting  $D_t u = \omega_a Ju$, we obtain
\begin{align*}
	\left<\dot\gamma, u\right>D_t u - \left<\dot\gamma, D_t u\right>u &= \omega_a \left[ \left<\dot\gamma, u\right>Ju - \left<\dot\gamma, Ju\right> u \right] \\
	&= \omega_a J \left[ \left<\dot\gamma, u\right>u + \left<\dot\gamma, Ju\right> Ju \right] \\
	&= \omega_a J \dot\gamma,
\end{align*}
since $(u,Ju)$ is an orthonormal basis. It follows that
\begin{align}\label{rotvar}
	\frac12 \frac{d}{ds} |D_t u_s|^2 = \left<D_t u, D_t w\right> + \left<V, \omega_a K J \dot\gamma \right>.
\end{align}

Combining \eqref{CMvar} and \eqref{rotvar} and integrating by parts, we obtain
\[
	\left.\frac{d}{ds} E(\gamma_s,u_s)\right|_{s=0} = \int_a^b \left( \big<V, I_a \omega_a K J \dot\gamma - m D_t \dot\gamma \big> - I_a \left<w, D_t^2 u\right> \right)dt
\]
as long as the variations $V$ and $w$ vanish at the endpoints $t = a,b$. Since $w$ and $V$ are arbitrary, the critical points must satisfy
\begin{align*}
	m D_t \dot\gamma &= I_a \omega_a K J \dot\gamma \\
	D_t^2 u &= 0.
\end{align*}
The first equation is precisely \eqref{eqn:intrinsic}, with $L = I_a \omega_a$. To complete the proof of Theorem \ref{thm:coordinate.free}, we just need to show that $\omega_a$ is constant. This is the case because $|\omega_a| = |D_t u|$   and
\[
	\frac{d}{dt} |D_t u|^2 = 2 \left<D_t^2 u, D_t u\right> = 0.
\]

\paragraph {\bf Acknowledgements} The authors are grateful for the many helpful comments by the referee, in particular for stimulating us to add an application, resulting in Section 2.3 on the Lagrange top. 

\bibliographystyle{plain}
\bibliography{top}
 \end{document}